\documentclass[12pt]{amsart}

\usepackage{amscd}

\advance\hoffset by -0.5 truecm

\usepackage{amssymb}

\newtheorem{Theorem}{Theorem}[section]

\newtheorem{Lemma}[Theorem]{Lemma}

\newtheorem{Corollary}[Theorem]{Corollary}

\def\V{\mbox{Var}}

\newcommand{\QED}{{\hfill$\Box$\medskip}}

\def\R\re

\def\V{\bf V}

\def \re{{\mathbb R}}

\def \0{\lambda_{0}}

\begin{document}

\title[Isoperimetric regions and Yamabe constants]{Isoperimetric
  regions in spherical cones and Yamabe constants of $M \times S^1$}

\author[J. Petean]{Jimmy Petean}

\address{CIMAT, A.P. 402, 36000, Guanajuato. Gto., M\'{e}xico.}
\email{jimmy@cimat.mx}

\thanks{J. Petean is supported by grant 46274-E of CONACYT}



\begin{abstract} We study isoperimetric regions on Riemannian
manifolds of the form $(M^n \times (0 ,\pi ), \sin^2 (t) g + dt^2 )$
where $g$ is a metric of positive Ricci curvature
$\geq n-1$. When $g$ is an Einstein metric we use this
to compute  the Yamabe constant of $(M\times \re , g+ dt^2 )$
and so to obtain lower bounds for the 
Yamabe invariant of $M\times S^1$.
\end{abstract}

\maketitle

\section{Introduction}

Given a closed Riemannian manifold $(M,g)$ we consider the
conformal class of the metric $g$, $[g]$. The Yamabe 
constant of $[g]$,  $Y(M,[g])$, is the
infimum of the normalized total scalar curvature functional on
the conformal class. Namely,

$$Y(M,[g])= \inf_{h\in [g]} 
\frac{\int {\bf s}_h \ dvol(h)}{(Vol(M,h))^{\frac{n-2}{n}}},$$

\noindent
where ${\bf s}_h$ denotes the scalar curvature of the metric $h$
and $dvol(h)$ its volume element.

If one writes metrics conformal to $g$ as $h=f^{4/(n-2)} \ g$,
one obtains the expression

$$Y(M,[g])= \inf_{f\in C^{\infty} (M)} 
\frac{\int ( \  4a_n {\| \nabla f \|}_g^2 + f^2 {\bf s}_g \ ) 
\  dvol(g)}{{\| f\|}_{p_n}^2},$$

\noindent
where $a_n =4(n-1)/(n-2) $ and $p_n =2n/(n-2)$. It is a fundamental
result 
on the subject that the infimum is actually achieved
(\cite{Yamabe, Trudinger, Aubin, Schoen}). The functions $f$ achieving
the infimum are called {\it Yamabe functions} and the corresponding metrics
$f^{4/(n-2)} \ g$ are called {\it Yamabe metrics}. Since the critical points 
of the total scalar curvature functional restricted to a conformal 
class of metrics are precisely the metrics of constant scalar
curvature in the conformal class, Yamabe metrics are metrics of
constant scalar curvature.

It is well known that by considering functions supported in a small
normal neighborhood of a point one can prove that 
$Y(M^n,[g]) \leq Y(S^n ,[g_0 ])$, where $g_0$ is the round metric
of radius one on the sphere and $(M^n ,g)$ is any closed n-dimensional 
Riemannian manifold (\cite{Aubin}). 
We will use the notation $Y_n = Y(S^n ,[g_0 ])$ and 
$V_n =Vol(S^n ,g_0 )$. Therefore $Y_n =n(n-1)V_n^{\frac{2}{n}}$.

Then one 
defines the {\it Yamabe invariant} of a closed manifold $M$ 
\cite{Kobayashi, Schoen2} as

$$Y(M)=\sup_g Y(M,[g]) \leq  Y_n .$$

It follows that $Y(M)$ is positive if and only if $M$ admits a
metric of positive scalar curvature. Moreover, the sign of 
$Y(M)$ determines the technical difficulties in understanding 
the invariant. When the Yamabe constant of a conformal class
is non-positive there is a unique metric  (up to multiplication
by a positive constant) of constant scalar curvature in the
conformal class and if $g$ is any metric in the conformal 
class, the Yamabe constant is bounded from below by
$(\inf_M {\bf s}_g  ) \ (Vol(M,g))^{2/n}$. This can be used for instance
to study the behavior of the invariant under surgery and so to
obtain information using cobordism theory \cite{Yun, Petean, Botvinnik}.
Note also that in the non-positive case the Yamabe invariant 
coincides with Perelman's invariant \cite{Ishida}.
The previous estimate  is no longer true
in the positive case, but one does get a lower bound in the case of
positive Ricci curvature by a theorem of S. Ilias: 
if $Ricci(g)\geq \lambda g $ 
($\lambda >0$) then $Y(M,[g]) \geq n \lambda (Vol(M,g))^{2/n}$
(\cite{Ilias}). Then in order to use this inequality to
find lower bounds on the Yamabe invariant of a closed 
manifold $M$ one would try to maximize the volume of the manifold
under some positive lower bound of the Ricci curvature. 
Namely, if one denotes ${\bf Rv} (M)= \sup \{ Vol(M,g): Ricci(g)\geq
(n-1) g \} $ then one gets $Y(M) \geq n(n-1) ({\bf Rv} (M))^{2/n}$ 
(one should define ${\bf Rv} (M) =0$ if $M$ does not admit 
a metric of positive Ricci curvature). Very little is known 
about the invariant ${\bf Rv} (M)$. Of course, Bishop's inequality 
tells us that for any n-dimensional closed manifold 
${\bf Rv} (M^n) \leq {\bf Rv} (S^n )$ 
(which is of course attained by the volume
of the metric of constant sectional curvature 1). Moreover,
G. Perelman \cite{Perelman} proved that there is a
constant $\delta =\delta_n >0$ such that if  ${\bf Rv} (M) \geq
{\bf Rv} (S^n ) -\delta_n $ then 
$M$ is homeomorphic to $S^n$. Beyond this, results on
${\bf Rv} (M)$ have been obtained by computing Yamabe invariants, so
for instance ${\bf Rv} ({\bf CP}^2 )= 2 \pi^2 $
(achieved by the Fubini-Study 
metric as shown by C. LeBrun \cite{Lebrun} and M. Gursky and C. 
LeBrun \cite{Gursky}) and ${\bf Rv} ({\bf RP}^3) = \pi^2$ (achieved by the 
metric of constant sectional curvature as shown by H. Bray and 
A. Neves \cite{Bray}).

Of course, there is no hope to apply the previous comments directly
when the fundamental group of $M$ is infinite. Nevertheless it
seems that even in this case the Yamabe invariant is
realized by conformal classes of metrics which maximize volume
with a fixed positive lower bound on the Ricci curvature
``in certain sense''. The standard example is $S^{n-1} \times 
S^1$. The fact that $Y(S^n \times S^1 ) =Y_{n+1}$ is one
of the first things we learned about the Yamabe invariant
\cite{Kobayashi, Schoen2}. One way to see this is as follows:
first one notes that $\lim_{T\rightarrow \infty} 
Y(S^n \times S^1 ,[g_0 + T^2 dt^2 ])=
Y(S^n \times \re, [g_0 + dt^2 ])$ \cite{Akutagawa} 
(the Yamabe constant for a non-compact Riemannian manifold
is computed as the infimum of the Yamabe functional over
compactly supported functions).
But the Yamabe
function for $g_0 + dt^2$ is precisely the conformal factor
between $S^n \times \re$ and $S^{n+1} -\{ S, N \}$. Therefore 
one can think of $Y(S^n \times S^1 ) =Y_{n+1}$ as realized 
by the positive
Einstein metric on $S^{n+1} -\{ S, N \} $. We will see in this
article that a similar situation occurs for any closed positive
Einstein manifold $(M,g)$  (although we only get the lower
bound for the invariant). 

\vspace{.3cm}

Let $(N,h)$ be a closed Riemannian manifold. An 
{\it isoperimetric region}
is an open subset $U$ with boundary $\partial U$ such that 
$\partial U$ minimizes area among hypersurfaces bounding a
region of volume $Vol(U)$. Given any positive number $s$, 
$s<Vol(N,h)$, there exists an isoperimetric region of
volume $s$. Its boundary is a stable constant mean curvature 
hypersurface with some singularities of codimension at least 7.
Of course one does not need a closed Riemannian manifold
to consider isoperimetric regions, apriori one only
needs to be able to compute volumes of open subsets and areas
of hypersurfaces. One defines the {\it isoperimetric function}
of $(N,h)$ as $I_h :(0,1) \rightarrow \re>0$ by

$$I_h (\beta) =\inf \{ Vol(\partial U)/Vol(N,h) : 
Vol(U,h) = \beta Vol(N,h) \},$$

\noindent
where $Vol(\partial U)$ is measured with the Riemannian metric 
induced by $h$ (on the non-singular part of $\partial U$).

Given a closed Riemannian manifold $(M,g)$ we will call
the {\it spherical cone} on $M$ the space $X$ obtained collapsing
$M \times \{0 \} $ and $M\times \{ \pi \}$ in
$M\times [0,\pi ]$ to points $S$ and $N$ (the vertices)
with the metric ${\bf g} =\sin^2 (t)g + dt^2$ 
(which is a Riemannian metric on $X-\{ S,N \}$). Now if
$Ricci(g) \geq (n-1) g$ one can see that $Ricci({\bf g})
\geq n{\bf g}$. One should compare this with the Euclidean cones 
considered by F. Morgan and M. Ritor\'{e} in \cite{Morgan}:
$\hat{g} =t^2 g + dt^2$ for which $Ricci(g) \geq (n-1)g $
implies that $Ricci(\hat{g}) \geq 0$. The importance of these
spherical cones for the study of Yamabe constants is that 
if one takes out the vertices the corresponding (non-complete)
Riemannian manifold is conformal to 
$M\times \re$. But using the (warped product version) of the
Ros Product Theorem \cite[Proposition 3.6]{Ros}  (see 
\cite[Section 3]{Morgan2}) and the Levy-Gromov isoperimetric
inequality \cite{Gromov} one can understand isoperimetric 
regions in these spherical cones. Namely,

\begin{Theorem} Let $(M^n,g)$ be a compact manifold with 
Ricci curvature $Ricci(g) \geq (n-1)g$. Let $(X,{\bf g})$ be
its spherical cone. Then geodesic balls around any of the 
vertices are isoperimetric.
\end{Theorem}

But now, since the spherical cone over $(M,g)$ 
is conformal to $(M\times \re ,
g+ dt^2 )$ we can use the previous result 
and symmetrization of a function with respect to the
geodesic balls centered at a vertex to prove:

\begin{Theorem} Let $(M,g)$ be a closed Riemannian manifold of  
positive Ricci curvature, $Ricci(g) \geq (n-1)g$ and volume $V$. 
Then 

$$Y(M\times \re ,[g+dt^2 ]) \geq 
(V/V_n )^{\frac{2}{n+1}} \  Y_{n+1} .$$
\end{Theorem}

\vspace{.2cm}

As we mentioned before one of the differences between the positive
and non-positive cases in the study of the Yamabe constant is
the non-uniqueness of constant scalar curvature metrics on
a conformal class with positive Yamabe constant. And the simplest
family of examples of non-uniqueness comes from Riemannian 
products. If $(M,g)$ and $(N^n ,h)$ are closed Riemannian manifolds
of constant scalar curvature and ${\bf s}_g$ is positive then
for small $\delta >0$, $\delta g + h$ is a constant scalar
curvature metric on $M \times  N$ which cannot be a Yamabe
metric. If $(M,g)$ is Einstein and $Y(M)=Y(M,[g])$ it seems
reasonable that $Y(M\times N)= \lim_{\delta \rightarrow 0} 
Y(M\times N ,[ \delta g + h ])$.
Moreover as it is shown in \cite{Akutagawa}

$$ \lim Y(M\times N , [\delta g + h ]) =Y(M\times \re^n,[ g+ dt^2 ]).$$

The only case which is well understood is when $M=S^n$ and $N=S^1$.
Here every Yamabe function is a function of the $S^1$-factor
\cite{Schoen2} and the Yamabe function for $(S^n \times \re , g_0 +
dt^2 )$ is the factor which makes $S^n\times \re$ conformal to
$S^{n+1} -\{ S, N \}$. It seems possible that under
certain conditions on $(M,g)$ the Yamabe functions of 
$(M \times \re^n , g+dt^2 )$ depend only on the second 
variable. The best case scenario would be that this is true
if $g$ is a Yamabe metric but it seems more attainable the
case when $g$ is Einstein. It is a corollary to the previous
theorem that this is actually true in the case $n=1$. Namely,
using the notation (as in \cite{Akutagawa}) 
$Y_N (M\times N , g +h)$
to denote the infimum of the $(g+h)$-Yamabe functional restricted
to functions of the $N$-factor we have:

\begin{Corollary} Let $(M^n,g)$ be a closed positive Einstein manifold 
with Ricci curvature $Ricci(g)=(n-1)g$. Then

$$Y(M\times \re , [g+ dt^2 ])=Y_{\re}(M\times \re , g+ dt^2 )=
{\left( \frac{V}{V_n} \right) }^{\frac{2}{n+1}} \ Y_{n+1}.$$ 

\end{Corollary}

\vspace{.3cm}

As $Y(M\times \re , [g+ dt^2 ]) = \lim_{T\rightarrow \infty }
Y(M\times S^1 ,[g+T dt^2 ])$ it also follows from Theorem 1.2
that:

\begin{Corollary} If $(M^n ,g)$ is a closed Einstein manifold
with $Ricci(g) = (n-1)g$ and volume $V$ then

$$Y(M\times S^1) \geq (V/V_n )^{\frac{2}{n+1}} \  Y_{n+1} .$$

\end{Corollary}

\vspace{.3cm}

So for example using the product metric we get

$$Y(S^2 \times S^2 \times S^1 )\geq  {\left(\frac{2}{3}
  \right)}^{(2/5)}  \ Y_5  $$

\noindent
and using the Fubini-Study metric we get

$$Y({\bf CP}^2 \times S^1 ) \geq {\left(\frac{3}{4} \right)}^{(2/5)} 
\ Y_5 .$$

\vspace{.4cm}

{\it Acknowledgements:} The author would like to  thank 
Manuel Ritor\'{e}, Kazuo Akutagawa and Frank Morgan
for several useful comments on the first drafts of 
this manuscript.

\section{Isoperimetric regions in spherical cones}

As we mentioned in the introduction, the isoperimetric
problem for spherical cones (over manifolds with 
Ricci curvature $\geq n-1$)  is understood using
the Levy-Gromov isoperimetric inequality 
(to compare the isoperimetric functions of
$M$ and of $S^n$) and the Ros Product Theorem for  warped products
(to compare then the isoperimetric functions of 
the spherical cone over $M$ to the isoperimetric function
of $S^{n+1}$). 
See for example section 3 of \cite{Morgan2}
(in particular {\bf 3.2} and the remark after it). For the
reader familiar with isoperimetric problems, this should be
enough to understand Theorem 1.1. In this section, for the
convenience of the reader, we will
give a brief outline on these issues. We will mostly 
discuss and follow section 3 of \cite{Ros} and ideas
in \cite{Morgan, Montiel} which we think might be useful in
dealing with other problems arising from the study of Yamabe 
constants.

Let $(M^n ,g)$ be a closed Riemannian manifold
of volume $V$ and Ricci curvature $Ricci(g) \geq (n-1)g$. 
We will
consider $(X^{n+1}, \bf{g}) $ where as a topological space $X$ is the 
suspension of $M$ ($X=M\times [0,\pi ]$ with $M\times \{ 0 \}$ and
$M\times \{ \pi \}$ identified to points $S$ and $N$) 
and $\bf{g}$ $ =\sin^2 (t) \ g \ +
dt^2$. Of course $X$ is not a manifold (except when $M$ is $S^n$) and
$\bf{g} $ is a Riemannian metric only on $X-\{ S,N \}$. 

The following is a standard result in geometric measure theory.

$\bf{Theorem:}$  For any positive number $r< Vol(x)$ there exists 
an isoperimetric open subset $U$ of $X$ of volume $r$. Moreover
$\partial U$ is a smooth stable constant mean curvature
hypersurface of $X$ except for a singular piece $\partial_1 U$
which consists of (possibly)
$S$, $N$, and a subset of codimension at least 7.

Let us call $\partial_0 U$ the regular part of $\partial U$,
$\partial_0 U= \partial U - \partial_1 U$. Let
$X_t$, $t\in (-\varepsilon ,\varepsilon )$, 
be a variation of  $\partial_0 U$ such that the 
volume of the enclosed region $U_t$ remains constant. 
Let $\lambda (t)$ be the area of $X_t$. Then $\lambda '(0) =0$
and $\lambda ''(0) \geq 0$. The first condition is satisfied
by hypersurfaces of constant mean curvature and the ones 
satisfying the second condition are called ${\it stable}$.
If $N$ denotes a  normal 
vector field to the hypersurface then variations are obtained 
by picking a function $h$ with compact support on $\partial_0 U$ and
moving $\partial_0 U$ in the direction of  $h \ N$. Then
we have that if the mean of $h$  on $\partial_0 U$
is 0 then  $\lambda_h '(0) =0$
 $\lambda_h ''(0) \geq 0$. This last condition is written as

$$Q(h,h)=-\int_{\partial_0 U} h(\Delta h + (Ricci (N,N) +
\sigma^2 )h ) dvol(\partial_0 U) \geq 0.$$

\noindent
Here we consider $\partial_0 U$ as a Riemannian manifold
(with the induced metric) and  use the corresponding Laplacian
and volume element. $\sigma^2$ is the square of the norm of the second
fundamental form. 
This was worked out by J. L. Barbosa, M. do Carmo and 
J. Eschenburg in \cite{Barbosa,
doCarmo}. As we said before, the function $h$ 
should apriori have compact support 
in $\partial_0 U$ but as shown by F. Morgan and M. Ritor\'{e} 
\cite[Lemma 3.3]{Morgan} it is enough that $h$ is bounded
and $h\in L^2 (\partial_0 U)$. This is important in order to study
stable constant mean curvature surfaces on a space like $X$ because
$X$ admits what is called a ${\it conformal}$ vector field $V=
\sin (t) \partial /\partial t$ and the function $h$ one wants to
consider is $h=div (V-{\bf g}(V,N) \ N )$ where $N$ is the unit
normal to the hypersurface (and then $h$ is the divergence of
the tangencial part of $V$). This has been used for instance in 
\cite{Montiel,Morgan} to classify stable constant mean curvature
hypersurfaces in Riemannian manifolds with a conformal vector field.
When the hypersurface is smooth this function $h$ has mean 0 by 
the divergence theorem and one can apply the stability condition. 
But when the hypersurface has singularities one would apriori need 
the function $h$ to have compact support on the regular part. This
was done by   F. Morgan and M. Ritor\'{e}  in 
\cite[Lemma 3.3]{Morgan}.

We want to prove that the geodesic balls around $S$ are
isoperimetric. One could try to apply the techniques of
Morgan and Ritor\'{e} in \cite{Morgan} and see that they are 
the only stable constant mean 
curvature hypersurfaces in $X$. This should be possible, and 
actually it might be necessary to deal with isoperimetric regions 
of more general singular spaces that appear naturally in the study of 
Yamabe constants of Riemannian products.
But in this case we will instead
take a more direct approach using the Levy-Gromov 
isoperimetric inequality \cite{Gromov} and Ros Product Theorem 
\cite{Ros}.

\vspace{.3cm}

The sketch of the proof is as follows: First one has to note that 
geodesic balls centered at the vertices {\it produce} the same 
isoperimetric function as the one of the round sphere. Therefore
to prove that geodesic balls around the vertices are isoperimetric 
is equivalent to prove that the isoperimetric function of ${\bf g}$ 
is bounded from below by the isoperimetric function of $g_0$. To
do this, given any open subset $U$ of $X$ one considers 
its symmetrization
$U^s \subset S^{n+1}$, so the the {\it slices} of $U^s$ are geodesic
balls with the same normalized volumes as the slices of $U$.  Then
by the Levy-Gromov isoperimetric inequality we can compare the
normalized areas of the boundaries of the slices. We have to 
prove that the normalized area of $\partial U^s$ is at most the
normalized area of $\partial U$. 
This follows from 
the warped product version of \cite[Proposition 3.6]{Ros}. We will
give an outline following Ros' proof for the Riemannian product case. 
We will use the notion of Minkowski
content. This is the bulk of the proof and we will divide it into 
Lemma 2.1, Lemma 2.2 and Lemma 2.3.

\vspace{.3cm}

{\it Proof of Theorem 1.1 :} 
Let $U\subset X$ be a closed  subset.
For any $t\in (0,\pi )$ let 

$$U_t =U \cap (M\times \{ t \} ) .$$

Fix any point $E\in S^n$ and let $(U^s )_t$ be the geodesic ball
centered at $E$ with volume 

$$Vol((U^s )_t , g_0 ) = \frac{V_n}{V} \ Vol(U_t ,g).$$ 

\noindent
(recall that $V=Vol(M,g)$ and $V_n = Vol(S^n ,g_0 )$).
Let $U^s
\subset S^{n+1}$ be the corresponding subset (i.e. we consider
$S^{n+1} -\{ S,N \}$ as $S^n \times (0,\pi )$ and $U^s$ is
such that $U^s \cap (S^n \times \{ t \}) $ =$(U^s )_t$.
One might add
$S$ and/or $N$ to make $U^s$ closed and connected). Note
that one can write $(U^s )_t = (U_t )^s = U_t^s$ as long as there
is no confusion (or no difference) on whether  we are considering 
it as a subset of $S^n$ or as a subset of $S^{n+1}$.

Now 

$$Vol(U)=\int_0^{\pi} \sin^n (t) \ Vol(U_t ,g) \ dt $$

$$= \frac{V}{V_n} \int_0^{\pi} \sin^n (t) \ Vol((U^s )_t ,g_0 ) \ dt 
= \frac{V}{V_n} Vol(U^s ,g_0 ).$$

Also if $B(r) =M\times [0,r]$ (the geodesic ball of radius
$r$ centered at the vertex at 0) then

$$Vol(B(r))=\int_0^r \sin^n (t) V dt = \frac{V}{V_n} 
\int_0^r \sin^n (t) V_n dt = \frac{V}{V_n} Vol (B_0 (r))  \ \ (1)$$

\noindent
where $B_0 (r)$ is the geodesic ball of radius $r$ in the
round sphere. And 

$$Vol(\partial B(r))=\sin^n (r) V =\frac{V}{V_n} 
Vol(\partial B_0 (r)) \ \ \ \ \ \ \ \ \ \ \ \ \ \ 
\ \ \ \ \ \ \ \ \ \ \ \ \ \ \ \ \ \ \ \ (2).$$

Formulas (1) and (2) tell us that the geodesic balls around the 
vertices in $X$ produce the same isoperimetric function as 
the round metric $g_0$. Therefore given any open subset $U \subset
X$ we want to compare the area of $\partial U$ with the area
of the boundary  
of the geodesic ball in $S^{n+1}$ with the same normalized volume
as $U$.

\vspace{.3cm}

Given a closed set $W$ let $B(W,r)$ be the set of points at distance
at most $r$ from $W$. Then one considers the {\it Minkowski content}
of $W$, 

$$\mu^+ (W) = \liminf \frac{Vol (B(W,r) ) -Vol(W)}{r}.$$

\noindent 
If $W$ is a smooth submanifold with boundary then 
$\mu^+ (W) = Vol (\partial W)$. And this is still true if the
boundary has singularities of codimension $\geq 2$ (and finite 
codimension 1 Hausdorff measure).
 
The Riemannian measure on $(S^n ,g_0 )$, normalized to be a
probability measure is what is called a {\it model measure}:
if $D^t$, $t\in (0,1)$ is the family of geodesic balls 
(with volume $Vol(D^t )=t$) centered at some fixed point then  
they are 
isoperimetric regions which are ordered by volume and such 
that for any $t$, $ B(D^t ,r) =D^{t'}$ for some $t'$. 
See \cite[Section 3.2]{Ros}. The following result follows 
directly from the 
Levy-Gromov isoperimetric inequality \cite[Appendix C]{Gromov}
and \cite[Proposition 3.5]{Ros} (see the lemma in 
\cite[page 77]{Morgan3} for a more elementary proof and point of view
on \cite[Proposition 3.5]{Ros}).

\begin{Lemma}: Let $(M,g)$ be a closed Riemannian manifold
of volume $V$ and Ricci curvature $Ricci(g) \geq (n-1) g$. 
For any nonempty closed subset $\Omega \subset M$ and
any $r\geq 0$ if $B_{\Omega}$ is a geodesic ball in
$(S^n , g_0 )$ with volume $Vol(B_{\Omega})=(V_n /V)
Vol(\Omega )$ then $Vol(B(B_{\Omega} ,r))  \leq 
(V_n /V) Vol(B(\Omega ,r))$.
\end{Lemma}

\begin{proof} Given any closed Riemannian manifold $(M,g)$, 
dividing the
Riemannian measure by the volume one obtains a probability 
measure which we will denote $\mu_g$. 
As we said before, the round metric on the sphere
gives a model measure $\mu_{g_0}$. On the other hand the Levy-Gromov
isoperimetric inequality \cite{Gromov}
says that $I_{\mu_g} \geq I_{\mu_{g_0}}$. 
The definition of $B_{\Omega}$ says that $\mu_g (\Omega )=\mu_{g_0}
(B_{\Omega})$ and what we want to prove is that $\mu_g (B(\Omega ,r))
\geq \mu_{g_0}
(B(B_{\Omega} ,r)  )$ .
Therefore the
statement of the lemma is precisely \cite[Proposition 3.5]{Ros}.

\end{proof}

Fix a positive constant $\lambda$. Note that the previous lemma
remains unchanged if we replace $g$ and $g_0$ by $\lambda g$
and $\lambda g_0$: the correspondence $\Omega \rightarrow
B_{\Omega}$ is the same and $\mu_{\lambda g} = \mu_g$.

\begin{Lemma} For any $t_0 \in (0,\pi )$ 
$B((U^s )_{t_0} ,r) \subset (B(U_{t_0} ,r ))^s $.
\end{Lemma}

\begin{proof}  First note that the distance from a point 
$(x,t) \in X$ to a vertex depends only on $t$ and not on $x$ 
(or even on $X$). Therefore if $r$ is greater than the
distance $\delta$ between $t_0$ and $0$ or $\pi$ 
then both sets in the lemma
will contain a geodesic ball of radius $r-\delta$ around the 
corresponding vertex. 

Also observe that the distance between points $(x,t_0 )$ and
$(y,t)$ depends only on the distance between $x$ and $y$
(and $t$, $t_0$, and the function in the warped product, 
which in this case is $\sin$) but not on $x, y$ or $X$.
In particular for any $t$ so that $|t-t_0 |<r$,  
$(B((U^s )_{t_0} ,r) )_t$ is a geodesic ball. 

We have to prove that for any $t$ 

$$(B((U^s )_{t_0} ,r) )_t \subset ((B(U_{t_0} ,r ))^s )_t.$$

\noindent
But since they are both geodesic balls centered at the same point
it is enough to prove that the volume of the subset on the left is
less than or equal to the volume of the subset on the right. 
By the definition of symmetrization the normalized volume of
$ ((B(U_{t_0} ,r ))^s )_t$ is equal to the normalized volume of 
$(B(U_{t_0} ,r ))_t$. But from the previous comment there exist
$\rho >0$ such that, considered as subsets of $M$,  

$$(B(U_{t_0} ,r ))_t = B(U_{t_0} ,\rho )$$

\noindent
and, as subsets of $S^n$,

$$(B((U^s )_{t_0} ,r) )_t =B(U^s_{t_0} ,\rho ).$$

The lemma then follows from Lemma 2.1 (and the comments after it). 

\end{proof}

Now for any closed subset $U\subset X$ let $B_U$ be a  
geodesic ball in $(S^{n+1} ,g_0 )$ with volume 
$Vol(B_U ,g_0 )= (V_n /V)
Vol(U,{\bf g})$. Since geodesic balls in round spheres are isoperimetric
(and $Vol(B_U ,g_0 )=Vol(U^s ,g_0 )$)
it follows that $Vol(\partial B_U )\leq \mu^+ (U^s )$.

\begin{Lemma} Given any closed set $U\subset X$, $\mu^+(U) 
\geq (V/V_n ) Vol(\partial B_U )$.
\end{Lemma}

\begin{proof}
Since   $(B(U,r) )^s$ is closed 
and  $B(U^s ,r)$ is the closure of
$\cup_{t\in (0,\pi )} \ B(U_t ^s ,r)$
we have from the previous lemma that

$$B(U^s ,r) \subset (B(U,r ) )^s .$$

 Then

$$Vol(\partial B_U )\leq \mu^+ (U^s ) 
=\liminf \frac{Vol(B(U^s ,r) ) - Vol (U^s )}{r}$$

$$\leq \liminf \frac{Vol((B(U ,r))^s  ) - Vol (U^s )}{r}$$

$$=(V_n /V)\liminf \frac{Vol(B(U,r) ) - Vol (U)}{r}
 =(V_n /V) \mu^+ (U) $$

\noindent
and the lemma follows.

\end{proof}

Now if we let $B_U^M$ be a geodesic ball around a vertex in $X$ 
with volume 

$$Vol(B_U^M ,{\bf g}) = Vol(U,{\bf g} ) = 
\frac{V}{V_n} Vol(B_U, g_0 )$$

\noindent 
then it follows from (1) and (2) in the beginning of the proof that 

$$Vol(\partial B_U^M ,{\bf g}) = \frac{V}{V_n} Vol(\partial B_U ,g_0 ).$$

\noindent
and so by Lemma 2.3

$$Vol(\partial B_U^M ,{\bf g}) \leq \mu^+ (U)$$

\noindent
and Theorem 1.1 is proved.

\QED

\section{The Yamabe constant of $M\times \re$}

Now assume that $g$ is a metric of positive Ricci curvature, 
$Ricci(g) \geq (n-1)g$ on $M$ and consider as before the
spherical cone $(X,{\bf g})$ with ${\bf g} =\sin^2 (t) g + dt^2$.
By a direct  
computation the sectional curvature of ${\bf g}$ is given by:

$$K_{{\bf g}} (v_i ,v_j )=\frac{K_g (v_i ,v_j )-\cos^2 (t)}{\sin^2 (t)}$$

$$K_{\bf g} (v_i ,\partial /\partial t)=1,$$

\noindent
for a $g$-orthonormal basis $\{ v_1 ,...,v_n \}$. And the Ricci
curvature is given by:

$$Ricci({\bf g}) (v_i ,\partial /\partial t )=0$$

$$Ricci({\bf g}) (v_i ,v_j )= Ricci(g) (v_i ,v_j ) - (n-1)\cos^2 (t)\delta_i^j
+\sin^2 (t) \delta_i^j$$

$$Ricci({\bf g}) (\partial_t ,\partial_t )=n.$$ 

Therefore  by picking $\{ v_1 ,...,v_n \}$ which diagonalizes $Ricci(g)$ one
easily sees that if $Ricci(g)\geq (n-1)g$ then $Ricci({\bf g})\geq n 
{\bf g}$. Moreover, if $g$ is an Einstein metric with Einstein 
constant $n-1$ the ${\bf g}$ is Einstein with Einstein constant $n$.

Let us recall that for non-compact 
Riemannian manifolds one defines
the Yamabe constant of a metric as the infimum of the Yamabe
functional of the metric
over smooth compactly supported functions (or functions
in $L_1^2$, of course). So for instance if $g$ is a Riemannian metric
on the closed manifold $M$ then

$$Y(M\times \re ,[g+dt^2 ]) =\inf_{f \in C^{\infty}_0 (M\times \re )}
\frac{\int_{M\times \re } \left( \ a_{n+1} {\| \nabla f \|}^2 + 
{\bf s}_g  \ f^2 \ \right) 
dvol(g+dt^2)}{
{\| f \|}_{p_{n+1}}^2 } .$$

\vspace{.2cm}

{\it Proof of Theorem 1.2 :}
We have a closed Riemannian manifold $(M^n ,g)$
such that $Ricci(g) \geq (n-1) g$. Let $f_0 (t)= \cosh^{-2} (t)$
and consider the diffeomorphism 

$$H: M \times (0, \pi )  \rightarrow M \times \re  $$

\noindent
given by $H(x,t)=(x,h_0 (t))$, where $h_0 :(0,\pi ) \rightarrow \re $
is the diffeomorphism defined by $h_0 (t) =cosh^{-1} ( (\sin
(t))^{-1})$
on $[\pi /2, \pi )$ and $h_0 (t)=-h_0 (\pi /2 -t)$ if
$t\in(0,\pi /2 )$. 

By a direct computation  $H^* ( f_0 (g+dt^2))= 
{\bf g}= \sin^2 (t) g +dt^2$ on $M\times (0,\pi )$.

Therefore by conformal invariance if we call $g_{f_0} = f_0 (g+dt^2)$

$$Y(M\times \re , [g+dt^2  ] ) 
=\inf_{ f \in C^{\infty}_0 (M\times \re )}
\frac{\int_{M\times \re } \left( \ a_{n+1} {\| \nabla f \|}_{g+dt^2}^2 + 
{\bf s}_g f^2 \right) \ dvol(g+dt^2)}{
{\| f \|}_{p_{n+1}}^2 } $$

$$=\inf_{f \in C^{\infty}_0 (M\times \re )}
\frac{\int_{M\times \re } \left( \ a_{n+1}  {\| \nabla f \|}^2_{g_{f_0}} 
+ {\bf s}_{g_{f_0}} f^2 \ \right) \  
  dvol(g_{f_0} )}{
{\| f \|}_{p_{n+1}}^2 } $$

$$=\inf_{f \in C^{\infty}_0 (M\times (0,\pi ))}
\frac{\int_{M\times (0,\pi ) } \ \left( a_{n+1} {\| \nabla f \|}^2_{\bf g} 
+ {\bf s}_{\bf g} 
f^2 \ \right) \ dvol({\bf g})}{
{\| f \|}_{p_{n+1}}^2 } =Y(M\times (0,\pi ),[{\bf g}]).$$

Now, as we showed in the previous section, $Ricci({\bf g})
\geq n$. Therefore ${\bf s}_{\bf g} \geq n(n+1)$. So we get 

$$Y(M\times \re , [g+dt^2 ]) \geq 
\inf_{f \in C^{\infty}_0 (M\times (0,\pi ))}
\frac{\int_{M\times (0,\pi ) } \ \left( a_{n+1} {\| \nabla f \|}^2_{\bf g} 
+ n(n+1) 
f^2 \ \right) \ dvol({\bf g})}{
{\| f \|}_{p_{n+1}}^2 }.$$

To compute the infimum one needs to consider only non-negative
functions. 
Now for any non-negative function 
$f \in C^{\infty}_0 (M\times (0,\pi ) \  )$ consider its symmetrization
$f_* :X \rightarrow \re_{\geq 0}$ defined by $f_* (S) =\sup f$ and
$f_* (x,t) =s$ if and only if $Vol(B(S,t), {\bf g} )=
Vol(\{ f > s \} ,{\bf g})$ (i.e. $f_*$ is a 
non-increasing function of $t$
and $Vol(\{ f_* > s \})=Vol(\{ f > s \}) $ for any $s$). 
It is inmediate that the $L^q$-norms of $f_*$ and $f$ are the
same for any $q$. Also, by the coarea formula

$$\int 
 \| \nabla f \|_{\bf g}^2 = \int_0^{\infty} 
\left( \int_{f^{-1}(t)} \| \nabla f \|_{\bf g} d\sigma_t \right) dt.$$

$$ \geq  \int_0^{\infty} (\mu (f^{-1} (t)))^2 
{\left( \int_{f^{-1}(t)} \| \nabla f \|_{\bf g}^{-1} d\sigma_t 
\right)}^{-1} \ dt$$

\noindent
by H\"{o}lder's inequality, where $d\sigma_t$ is the measure induced
by ${\bf g}$ on $\{ f^{-1} (t) \}$. But 

$$\int_{f^{-1}(t)} \| \nabla f \|_{\bf g}^{-1} d\sigma_t 
=-\frac{d}{dt} (\mu\{ f>t \})$$

$$=-\frac{d}{dt} (\mu\{ f_* >t \}) = 
\int_{f_*^{-1}(t)} \| \nabla f_* \|_{\bf g}^{-1} d\sigma_t $$

\noindent 
and since $f^{-1} (t) =\partial \{ f>t \}$ by Theorem 1.1 
we have $\mu (f^{-1} (t))\geq \mu (f_*^{-1} (t))$. Therefore

$$ \int_0^{\infty} (\mu (f^{-1} (t)))^2 
{\left( \int_{f^{-1}(t)} \| \nabla f \|_{\bf g}^{-1} d\sigma_t 
\right)}^{-1} \ dt$$

$$ \geq  \int_0^{\infty} (\mu (f_*^{-1} (t)))^2 
{\left( \int_{f_*^{-1}(t)} \| \nabla f_* \|_{\bf g}^{-1} d\sigma_t 
\right)}^{-1} \ dt $$

\noindent
(and since  $\| \nabla f_* \|_{\bf g}$ is constant along 
$f_*^{-1}(t)$ )

$$=\int_0^{\infty}\mu (f_*^{-1} (t)) \| \nabla f_* \|_{\bf g} \ dt$$

$$= \int_0^{\infty} 
\left( \int_{f_* ^{-1}(t)} \| \nabla f_* \|_{\bf g} d\sigma_t \right)
dt =\int 
 \| \nabla f_*  \|_{\bf g}^2 .$$

Considering $S^{n+1}$ as the spherical cone over $S^n$ we have 
the function $f^0_* : S^{n+1} \rightarrow \re_{\geq 0}$ which
corresponds to $f_*$. 

Then for all $s$ 

$$Vol (\{ f_*^0 >s \} ) = 
\left( \frac{V_n}{V} \right) \  Vol( \{ f_* >s \},$$

\noindent
and so for any $q$,

$$\int (f^0_*)^q dvol(g_0 ) = \left( \frac{V_n}{V} \right) 
\int (f_* )^q dvol({\bf g}).$$

Also for any $s\in (0,\pi )$

$$\mu ( (f_*^0 )^{-1} (s)) = \frac{V_n}{V} \mu (f_*^{-1} (s)),$$

\noindent
and since ${\| \nabla f_*^0 \|}_{g_0} = {\| \nabla f_* \| }_{\bf g}$
we have

$$  \int 
 \| \nabla f^0_*  \|_{g_0}^2 = \frac{V_n}{V} \int 
 \| \nabla f_*  \|_{\bf g}^2 .$$

We obtain

$$Y(M\times \re , [g+dt^2 ]) \geq 
\inf_{f \in C^{\infty}_0 (M\times (0,\pi ))}
\frac{\int_{M\times (0,\pi ) } a_{n+1} {\| \nabla f \|}^2_{\bf g} 
+ n(n+1) 
f^2 \ dvol({\bf g})}{
{\| f \|}_{p_{n+1}} ^2 }$$

$$\geq \inf_{f \in C^{\infty}_0 (M\times (0,\pi ))}
\frac{\int_{M\times (0,\pi ) } a_{n+1} {\| \nabla f_* \|}^2_{\bf g} 
+ n(n+1) 
f_*^2 \ dvol({\bf g})}{
{\| f_* \|}_{p_{n+1}}^2 }$$

$$={\left( \frac{V}{V_n} \right)}^{1-(2/p_{n+1})}
\inf_{f \in C^{\infty}_0 (M\times (0,\pi ))}
\frac{\int_{M\times (0,\pi ) } a_{n+1} {\| \nabla f^0_* \|}^2_{g_0} 
+ n(n+1) 
{f^0_*}^2 dvol({g_0})}{
{\| f^0_* \|}_{p_{n+1}}^2 }$$

$$ \geq 
{\left( \frac{V}{V_n} \right)}^{2/(n+1)} \  Y_{n+1}$$

This finishes the proof of Theorem 1.2.

\QED

{\it Proof of Corollary 1.3 :} Note that if
${\bf s}_g$ is constant $Y_{\re} (M \times \re , g +
dt^2)$
only depends on ${\bf s}_g$ and $V=Vol(M,g)$ Actually,

$$Y_{\re} (M\times \re ,g +dt^2 )=
\inf_{f\in C_0^{\infty} ( \re )} \frac{\int_{\re} \  a_{n+1} {\|\nabla f
    \|}^2_{dt^2} V
+ {\bf s}_g V f^2 \ dt^2}{(\int_{\re} f^p )^{2/p} \  V^{2/p}}$$

$$=V^{1-(2/p)}
\inf_{f\in C_0^{\infty} ( \re )} \frac{\int_{\re} \  a_{n+1} {\|\nabla f
    \|}^2_{dt^2} 
+ {\bf s}_g  f^2 \ dt^2}{(\int_{\re} f^p )^{2/p}}.$$

But as we said 

$$\inf_{f\in C_0^{\infty} ( \re )} \frac{\int_{\re} \  a_{n+1} {\|\nabla f
    \|}^2_{dt^2} 
+ {\bf s}_g  f^2 \ dt^2}{(\int_{\re} f^p )^{2/p}}$$

\noindent
is independent of $(M,g)$ and it is known to be equal to 
$Y_{n+1} V_n^{-2/(n+1)}$.  Corollary 1.3 then follows 
directly from Theorem 1.2.

\QED

\end{document}